\documentclass[12pt]{amsart}
\usepackage{amscd}
\usepackage[latin1]{inputenc}
\usepackage{amscd}
\usepackage{latexsym}
\usepackage{pstcol,pst-plot,pst-3d}
\usepackage{multicol}
\psset{unit=0.7cm,linewidth=0.8pt,arrowsize=2.5pt 4}

\newpsstyle{fatline}{linewidth=1.5pt}
\newpsstyle{fyp}{fillstyle=solid,fillcolor=verylight}
\definecolor{verylight}{gray}{0.97}
\definecolor{light}{gray}{0.9}
\definecolor{medium}{gray}{0.85}

\def\x{{\bold x}}

\def\1{{\mathbf 1}}
\def\0{{\mathbf 0}}

\def\opn#1#2{\def#1{\operatorname{#2}}} 
\opn\relint{relint}
\opn\height{ht}
\opn\tr{tr}
\opn\ini{in}
\opn\Spec{Spec} \opn\Supp{Supp} \opn\supp{supp} \opn\Sing{Sing}
\opn\Ass{Ass} \opn\Min{Min}\opn\Mon{Mon} \opn\dstab{dstab} \opn\astab{astab}
\opn\Syz{Syz} \opn\depth{depth}
\newtheorem{Theorem}{Theorem}[section]
\newtheorem{Lemma}[Theorem]{Lemma}
\newtheorem{Corollary}[Theorem]{Corollary}
\newtheorem{Proposition}[Theorem]{Proposition}
\newtheorem{Remark}[Theorem]{Remark}

\newtheorem{Example}[Theorem]{Example}

\newtheorem{Definition}[Theorem]{Definition}

\newtheorem{Question}[Theorem]{Question}
\begin{document}

\title{Strong persistence and associated primes of powers of monomial ideals}

\author{Amir Mafi and Hero Saremi}

\address{Amir Mafi, Department of Mathematics, University Of Kurdistan, P.O. Box: 416, Sanandaj, Iran.}
\email{A\_Mafi@ipm.ir}

\address{Hero Saremi, Department of Mathematics, Sanandaj Branch, Islamic Azad University, Sanandaj, Iran.}
\email{hero.saremi@gmail.com}

\begin{abstract}
Let $R=K[x_1,\ldots, x_n]$ be the polynomial ring in $n$ variables over a field $K$ and $I$ be a monomial ideal of degree $d\leq 2$.
We show that $(I^{k+1}:I)=I^k$ for all $k\geq 1$ and we disprove a motivation question that was asked by of Carlini, H\`a, Harbourne and Van Tuyl by providing of a counterexample. Also, by this counterexample, we give a negative answer to the question that the depth function of square-free monomial ideals are non-increasing.
\end{abstract}


\maketitle

\section{Introduction}
Throughout this paper, we assume that $R=K[x_1,\ldots,x_n]$ is the polynomial ring in $n$ variables over the field $K$ with the maximal ideal $\frak{m}=\langle x_1,\ldots,x_n\rangle$ and $I$ is a monomial ideal of $R$. Let $\Ass(I)$ be the set of associated primes of $R/I$. Brodmann \cite{B} showed that there exists an integer $k_0$ such that $\Ass(I^k)=\Ass(I^{k_0})$ for all $k\geq k_0$. This stable set of associated primes is denoted by $\Ass^{\infty}(I)$. Recently there have been several publications in \cite{ANR, BHR, BFT, CHHV, HM, HRV, KM, MN} about the set of associated primes of monomial ideals. An ideal $I$ is said to satisfy the {\it strong persistence property} if $(I^{k+1}:I)=I^k$ for all $k\geq 1$. Ratliff \cite{R} proved that $(I^{k+1}:I)=I^k$ for all large $k$.
There are some interesting classes of monomial ideals that satisfy the strong persistence property. These include normal ideals, edge ideals of graphs, vertex cover ideals of perfect graphs (or chordal graphs), polymatroidal ideals, vertex cover ideals of cycle graphs of odd orders, vertex cover ideals of wheel graphs of even orders and all square-free monomial ideals in $R$ with $n\leq 4$, see for details \cite{FHV1, FHV2, HRV, MMV, NKA,R, RT, V}.

An ideal $I$ is said to satisfy the {\it persistence property} if $\Ass(I^k)\subseteq\Ass(I^{k+1})$ for all $k\geq 1$. It is known that
the strong persistence property yields the persistence property, see \cite{HQ}.
Brodmann \cite{B1} showed that there exists an integer $k_0$ such that $\depth R/I^k=\depth R/I^{k_0}$ for all $k\geq k_0$. Herzog and Hibi \cite{HH1} proved that if $I$ is a graded ideal such that $I^k$ has a linear resolution for all $k\geq 1$, then $\depth R/I^k$ is a non-increasing function for all $k\geq 1$.
The following questions were raised in \cite{HH1, HRV, MV}:

\begin{Question}\label{Q1}

\par {\rm (i)}  Do all square-free monomial ideals satisfy the strong persistence property?
\par {\rm (ii)} Do all square-free monomial ideals satisfy the persistence property?
\par {\rm (iii)} Do all square-free monomial ideals have non-increasing depth function?
\par \noindent
\end{Question}

A remarkable example of a vertex cover ideal $I$ was given in \cite{KSS} with the property that $(I^4:I)\neq I^3$, $\Ass(I^3)\not\subseteq\Ass(I^4)$ and $\depth R/I^3\not\geq\depth R/I^4$. This gives a negative answer to the above questions.
This counterexample has the property that $\Ass(I^2)\subseteq\Ass(I^n)$ for all $n\geq 2$. After that Carlini, H\`a, Harbourne and Van Tuyl in \cite[Question 2.51]{CHHV} asked the following question:

\begin{Question}\label{Q2}
Let $I$ be any square-free monomial ideal. Does $\Ass(I^2)\subseteq\Ass(I^k)$ for all $k\geq 2$?
\end{Question}

In this paper, we disprove Question \ref{Q2}, in general, by providing a counterexample, and also we give some positive answers in some special cases. Furthermore, we show that the strong persistence property holds for all monomial ideals of degree $d\leq 2$. This is a generalization of \cite[Theorem 2.15]{MMV}.
For unexplained notation or terminology, we refer the reader to \cite{HH}.

\section{The results}

We start this section with the following definition:
\begin{Definition}
Let $I$ be an ideal of $R$. Then an element $x\in R$ is {\bf integral over} $I$, if there is an equation
\[ {x}^k+a_1{x}^{k-1}+\dots+a_{k-1}{x}+a_k=0,\]
with $a_i\in I^i$. The set of elements $\overline{I}$ in $R$ which are integral over $I$ is the integral closure of $I$. It is clear
 $I\subseteq \overline{I}\subseteq \sqrt{I}$. The ideal $I$ is integrally closed if $I=\overline{I}$, and $I$ is normal if all powers of $I$ are integrally closed.\\ It is clear that if $I$ is a square-free monomial ideal, then $I=\sqrt{I}$ and so $I$ is integrally closed.
\end{Definition}

The following result is known but for readers we give an easy proof:

\begin{Lemma}\label{L1}
Let $I$ be a monomial ideal of $R$. If $(I^{k+1}:I)=I^k$ for some $k\geq 1$, then $\Ass(I^k)\subseteq\Ass(I^{k+1})$.
In particular, if $I$ satisfies the strong persistence property, then $I$ satisfies the persistence property.
\end{Lemma}

\begin{proof}
Suppose $\frak{p}\in\Ass(I^k)=\Ass(I^{k+1}:I)$. Then $\frak{p}=((I^{k+1}:I):\alpha)$ for some monomial element $\alpha\in R$ and so  $\frak{p}=(I^{k+1}:{\alpha}I)$. Thus there is a monomial $u\in I$ such that $\frak{p}=(I^{k+1}:{\alpha}u)$ and hence $\frak{p}\in\Ass(I^{k+1})$.
Therefore $\Ass(I^k)\subseteq\Ass(I^{k+1})$. For the remaining proof, by hypothesis we have $(I^{k+1}:I)=I^k$ for all $k\geq 1$ and this implies $\Ass(I^k)\subseteq\Ass(I^{k+1})$ for all $k\geq 1$, as required.
\end{proof}

\begin{Lemma}(\cite[Lemma 11.27]{M})\label{L0}
Let $I$ be an ideal of $R$. Then, for all $n\geq m\geq 1$, $(\overline{I^n}:\overline{I^m})=(\overline{I^n}:I^m)=\overline{I^{n-m}}$.
\end{Lemma}

\begin{Remark}
Let $I$ be a square-free monomial ideal of $R$. Then, for all $k\geq 1$, \[I\subseteq(I^2:I)\subseteq(I^k:I^{k-1})\subseteq(\overline{I^k}:I^{k-1})=\overline{I}=I\] the first equality follows by Lemma \ref{L0}. Thus $(I^k:I^{k-1})=I$ and so by the same argument of Lemma \ref{L1} we have $\Ass(I)\subseteq\Ass(I^{k})$ for all $k\geq 1$.
\end{Remark}

\begin{Definition}
For a square-free monomial ideal $I=\langle\{x_{i1}\dots x_{in_i}\mid i=1,\dots,t\}\rangle$ of $R$ the {\bf Alexander dual} of $I$, denoted by $I^{\vee}$, is is the square-free monomial ideal $I^{\vee}=\cap_{i=1}^t\langle x_{i1},\dots, x_{in_i}\rangle$.
\end{Definition}

Following \cite{FHV} let $G$ be a finite simple graph on the vertex set $V=\{x_1,\ldots,x_n\}$ with edge set $E$. By identifying the vertices with the variables in the polynomial ring $R=K[x_1,\ldots, x_n]$, one can associate to $G$ a square-free quadratic monomial ideal $I(G)=\langle x_ix_j|~\{x_i,x_j\}\in E\rangle.$
The ideal $I(G)$ is called the {\it edge ideal} of $G$ which was first introduced by Villarreal \cite{V1}. The ideal $I(G)^{\vee}$ is referred to as the {\it cover ideal} because of the well-known fact that the generators of $I(G)^{\vee}$ correspond to vertex covers (see \cite{HHT}).

\begin{Theorem}(\cite[Theorem 3.2]{FHV})\label{T1}
Let $G$ be a finite simple graph and $J=I(G)^{\vee}$. Then the irredundant irreducible decomposition of $J^2$ is
\[J^2=\bigcap_{\{x_i,x_j\}\in E}\langle x_i,x_j\rangle^2\cap \bigcap_{\{x_{i_1},\ldots,x_{i_s}\} ~is ~an~ induced~ odd~ cycle}\langle x_{i_1}^2,\ldots,x_{i_s}^2\rangle. \]
\end{Theorem}

Simis, Vasconcelos and Villarreal in \cite[Proposition 6.6]{SVV} proved that if $I$ is an edge ideal of a finite simple graph, then $I^2$ is integrally closed. In the following proposition, we prove a similar result for the vertex cover ideal of an edge ideal.

\begin{Proposition}\label{P1}
Let $I$ be an edge ideal of a finite simple graph $G$  and $J$ be a vertex cover ideal of $I$. Then $\overline{J^2}=J^2$. In particular, $(J^3:J)=J^2$.
\end{Proposition}

\begin{proof}
It is enough to prove that $\Ass(\overline{J^2}/J^2)=\emptyset$. By Theorem \ref{T1}, $\Ass(R/J^2)=\{(x_i,x_j)\mid~ \{x_i,x_j\}\in E\}\cup\{(x_{i_1},\ldots,x_{i_s})\mid~$ $\{x_{i_1},\ldots,x_{i_s}\}$ is an induced odd cycle of $G$ $\}$. Since $\Ass(\overline{J^2}/J^2)\subseteq\Ass(R/J^2)$, it readily follows $(x_i,x_j)\not\in\Ass(\overline{J^2}/J^2)$ and so we may assume $\frak{p}=(x_{i_1},\ldots,x_{i_s})\in\Ass(\overline{J^2}/J^2)$.
Therefore $\overline{J^2_{\frak{p}}}/J^2_{\frak{p}}\neq 0$ and we may assume $J_{\frak{p}}$ is a vertex cover ideal of an odd cycle graph. By using
\cite[Theorem 1.10]{ANR} we have $\overline{J^2_{\frak{p}}}=J^2_{\frak{p}}$ and this is a contradiction. Thus $\Ass(\overline{J^2}/J^2)=\emptyset$ and so
$\overline{J^2}=J^2$. For the remaining of the proof, we have $J^2\subseteq(J^3:J)\subseteq(\overline{J^3}:J)=\overline{J^2}$ the last equality follows by Lemma \ref{L0}. Since $\overline{J^2}=J^2$, we have the result. This completes the proof.
\end{proof}

The following corollary gives a partially positive answer to Question \ref{Q2}.

\begin{Corollary}\label{C1}
Let $I$ be a square-free monomial ideal of $R$ such that $\overline{I^2}=I^2$. Then $\Ass(I^2)\subseteq\Ass(I^k)$ for all $k\geq 2$. In particular,
if $I$ is a vertex cover ideal of a finite simple graph then $\Ass(I^2)\subseteq\Ass(I^k)$ for all $k\geq 2$.
\end{Corollary}

\begin{proof}
By our hypothesis and Lemma \ref{L0}, we have $I^2\subseteq(I^k:I^{k-2})\subseteq(\overline{I^k}:I^{k-2})=I^2$ for all $k\geq 2$.
Therefore $(I^k:I^{k-2})=I^2$. Now, by using the same argument as Lemma \ref{L1} we have $\Ass(I^2)\subseteq\Ass(I^k)$ for all $k\geq 2$.
In particular, if $I$ is a vertex cover ideal of a finite simple graph, then by Proposition \ref{P1} we have the result, as required.
\end{proof}

By the following example we show that Question \ref{Q2} and also Question \ref{Q1} in general have negative answers:

\begin{Example}
Let $R=K[x_1,\ldots,x_7]$ and $I=
\langle x_1x_4x_5x_7,x_2x_3x_6,x_2x_3x_7,x_2x_4x_5,\\x_2x_4x_7,x_2x_5x_6,x_3x_4x_5,x_3x_4x_6,x_3x_5x_7,x_4x_6x_7,x_5x_6x_7\rangle$
be the square-free monomial ideal of $R$. By using Macaulay2 we have $\frak{m}\in\Ass(I^2)\setminus\Ass(I^3)$ and by the algorithm \cite{BHR} it follows $\frak{m}\notin\Ass^{\infty}(I)$. Therefore $\Ass(I^2)\not\subseteq\Ass(I^3)$, $(I^3:I)\neq I^2$, $\depth(R/I^3)\not\leq\depth(R/I^2)$ and $\Ass^{\infty}(I)\neq\bigcup_{k=1}^{\infty}\Ass(I^k)$.

\end{Example}

For the next result we introduce the following terminology: let $I=\langle u_1,\ldots,u_t\rangle$ be  a monomial ideal of $R$. Then we set $\supp(I)=\bigcup_{i=1}^t\supp(u_i)$ and $\deg(I)=\max\{\deg(u_i)\mid 1\leq i\leq t\}$, where $\supp(u)=\{x_i|~ u=x_1^{a_1}\ldots x_n^{a_n}, a_i\neq 0\}$
and $\deg(x_1^{a_1}\ldots x_s^{a_s})=\sum_{i=1}^s a_i$.

\begin{Lemma}\label{L2}
Let $\underline{x}=\langle x_1,\ldots,x_t\rangle$ and $J$ be a monomial ideal of $R$ such that $x_i\not\in\supp(J)$ for all $1\leq i\leq t$.
Then $(\underline{x}^{k+1}J:x_i)=\underline{x}^kJ$ for all $k\geq 1$.
\end{Lemma}

\begin{proof}
Suppose $\alpha\in(\underline{x}^{k+1}J:x_i)$ and so $\alpha x_i\in\underline{x}^{k+1}J$.
Since $\underline{x}^{k+1}J=(\underline{y}^{k+1}+x_i\underline{x}^k)J$ where $\underline{y}=\langle x_1,\ldots,x_{i-1},x_{i+1},\ldots,x_t\rangle$, it therefore follows that $\alpha x_i\in x_i\underline{x}^kJ$ and so $\alpha\in\underline{x}^kJ$.
Hence $(\underline{x}^{k+1}J:x_i)=\underline{x}^kJ$ for all $k\geq 1$, as required.
\end{proof}

\begin{Proposition}\label{P2}
Let $I$ be a monomial ideal of $R$ such that $\deg(I)\leq 2$. If $I=\langle x_1,\ldots,x_t\rangle +J$ such that $J$ is a monomial ideal of $R$ with $\deg(J)\leq 2$ and $x_i\notin\supp(J)$ for all $1\leq i\leq t$, then $(I^{k+1}:I)=I^k$ for all $k\geq 1$ and $1\leq i\leq t$.
\end{Proposition}

\begin{proof}
First we prove that $(I^{k+1}:x_i)=I^k$ for all $k\geq 1$ and $1\leq i\leq t$.
 It is clear $I^{k+1}=\sum_{j=0}^{k+1} \underline{x}^jJ^{k+1-j}$, where $\underline{x}=\langle x_1,\ldots,x_t\rangle$.
 It therefore follows $(I^{k+1}:x_i)=J^{k+1}+\sum_{j=1}^{k+1}(\underline{x}^jJ^{k+1-j}:x_i)$. Now by using Lemma \ref{L2} and \cite[Lemma 2.1]{Ma}
 we get $(I^{k+1}:x_i)=\sum_{j=0}^{k}\underline{x}^jJ^{k-j}=I^k$. Since $I^k\subseteq(I^{k+1}:I)\subseteq(I^{k+1}:\underline{x})\subseteq(I^{k+1}:x_i)=I^k$, we have $(I^{k+1}:I)=I^k$ for all $k\geq 1$, as required.

\end{proof}

The following result is a generalization of \cite[Theorem 2.15]{MMV}.
\begin{Theorem}
Let $I$ be a monomial ideal of $R$ such that $\deg(I)\leq 2$. Then $(I^{k+1}:I)=I^k$ for all $k\geq 1$. In particular, $\Ass(I^k)\subseteq\Ass(I^{k+1})$
for all $k\geq 1$.
\end{Theorem}
\begin{proof}
If $I$ is a monomial ideal of single degree $2$, then the result follows by \cite[Theorem 3]{RT}. Now, we may assume that $I=\langle x_1,\ldots,x_t\rangle +J$ such that $\deg(J)\leq 2$. Therefore by Proposition \ref{P2} we have $(I^{k+1}:I)=I^k$. The proof then follows by Lemma \ref{L1}.

\end{proof}

{\bf Acknowledgement:} We would like to thank Kamran Divaani-Aazar and Mehrdad Nasernejad for some helpful comments. Also, we would like to thank deeply grateful to the referee for the careful reading of the manuscript.


\begin{thebibliography}{99}

\bibitem{ANR} I. Al-Ayyoub, M. Nasernejad and L. G. Roberts, {\it Normality of cover ideals of graphs and normality under some operations}, Results Math., {\bf74}(2019), No. 140, 26 pp.

\bibitem{BHR} S. Bayati, J. Herzog and G. Rinaldo, {\it On the stable set of associated prime ideals of a monomial ideal}, Arch. Math., {\bf 98}(2012), 213-217.
\bibitem{BFT} E. Bela, G. Favacchio and N. Tran, {\it In the shadows of a hypergraph: looking for associated
primes of powers of square-free monomial ideals}, J. Algebr. Comb., {\bf 53}(2021), 11-29.

\bibitem{B} M. Brodmann, {\it Asymptotic stability of $\Ass(M/{I^nM})$}, Proc. Amer. Math. Soc., {\bf 74}(1979), 16-18.

\bibitem{B1} M. Brodmann, {\it The asymptotic nature of the analytic spread}, Math. Proc. Cambridge Philos. Soc., {\bf 86}(1979), 35-39.

\bibitem{CHHV}E. Carlini, H. T. H\`a, B. Harbourne and  A. Van Tuyl, {\it Ideals of powers and powers
of ideals, intersecting Algebra, Geometry, and Combinatorics}, Lect. Notes of the Unione Matematica Italiana, {\bf 7172}, Springer Nature Switzerland (2020).

\bibitem{FHV} C. A. Francisco, H. T. H\`a and  A. Van Tuyl, {\it Associated primes of monomial ideals and odd holes in graphs}, J. Algebr. Comb., {\bf 32}(2010), 287-301.
\bibitem{FHV1} C. A. Francisco, H. T. H\`a and  A. Van Tuyl, {\it A conjecture on critical graphs and connections to the persistence of associated primes}, Discrete Math., {\bf 310}(2010), 2176-2182.
\bibitem{FHV2} C. A. Francisco, H. T. H\`a nd  A. Van Tuyl, {\it  Colorings of hypergraphs, perfect graphs, and associated primes of powers of monomial ideals}, J. Algebra, {\bf 331}(2011), 224-242.

\bibitem{GS}D. R. Grayson and M. E. Stillman, {\it Macaulay 2, a software system for research in algebraic geometry}, Available at {http://www.math.uiuc.edu/Macaulay2/}.

\bibitem{HH1} J. Herzog and T. Hibi, {\it The depth of powers of an ideal}, J. Algebra, {\bf 291}(2005), 534-550.

\bibitem{HH} J. Herzog and T. Hibi, {\it Monomial ideals}, GTM., {\bf 260}, Springer, Berlin, (2011).
\bibitem{HHT} J. Herzog, T. Hibi and N. V. Trung, {\it Symbolic powers of monomial ideals and vertex cover algebras}, Adv. Math., {\bf 210}(2007), 304-322.
\bibitem{HM} J. Herzog and A. Mafi, {\it Stability properties of powers of ideals in regular local rings of small dimension}, Pacific J. Math., {\bf 295}(2018), 31-41.
\bibitem{HQ} J. Herzog and A. A. Qureshi, {\it Persistence and stability properties of powers of ideals}, J. Pure and Appl. Algebra, {\bf 219}(2015), 530-542.

\bibitem{HRV} J. Herzog, A. Rauf and M. Vladoiu, {\it The stable set of associated prime ideals of a polymatroidal ideal}, J. Algebr. Comb., {\bf 37}(2013), 289-312.
\bibitem{KSS} T. Kaiser, M. Stehlik and R. Skrekovski, {\it Replication in critical graphs and the persistence of monomial ideals}, J. Comb. Theory Ser. A, {\bf 123}(2014), 239-251.

\bibitem{KM} Sh. Karimi and A. Mafi, {\it On stability properties of powers of polymatroidal ideals}, Collect. Math., {\bf 70}(2019), 357-365.

\bibitem{Ma} A. Mafi, {\it Ratliff-Rush ideal and reduction numbers}, Comm. Algebra, {\bf 46}(2018), 1272-1276.

\bibitem{MN} A. Mafi and D. Naderi, {\it A note on stability properties of powers of polymatroidal ideals}, arXiv:2112.05918v1 [math.AC](2021).

\bibitem{MMV} J. Martinez-Bernal, S. Morey and R. H. Villarreal, {\it Associated primes of powers of edge ideals}, Collect. Math., {\bf 63}(2012), 361-374.

\bibitem{M} S. McAdam, {\it Asymptotic prime divisors}, Lect. Notes Math., {\bf 1023}, Springer-Verlag, New York, (1983).

\bibitem{MV} S. Morey and R. H. Villarreal, {\it Edge ideals: Algebraic and Combinatorial properties}, De Gruyter, Berlin, (2012), pp. 85-126.

\bibitem{NKA} M. Nasernejad, K. Khashyarmanesh and I. Al-Ayyoub, {\it Associated primes of powers of cover ideals under graph operations}, Comm. Algebra, {\bf 47}(2019), 1985-1996.
\bibitem{R} L. J. Ratliff, {\it On prime divisors of $I^n$, n large}, Michigan. Math. J., {\bf 23}(1976), 337-352.

\bibitem{RT} E. Reyes and J. Toledo, {\it On the strong persistence property for monomial ideals}, Bull. Math. Soc. Sci. Math. Roumanie, {\bf 80}(2018), 293-305.
\bibitem{SVV} A. Simis, V. V. Vasconcelos and R. H. Villarreal, {\it On the ideal theory of graphs}, J. Algebra, {\bf 167}(1994), 389-416.

\bibitem{V1}  R. H. Villarreal, {\it Cohen-Macaulay graphs}, Manuscripta Math., {\bf 66}(1990), 277-293.

\bibitem{V} R. H. Villarreal, {\it Rees algebras and polyhedral cones of ideals of vertex covers of perfect graphs}, J. Algebr. Comb., {\bf 27}(2008), 293-305.


\end{thebibliography}
\end{document}